\documentclass{article}

\title{Multiplicative formality of operads and \\
Sinha's spectral sequence for long knots}

\author{Syunji Moriya}
\date{}

\usepackage{setspace}
\usepackage{amsthm}
\usepackage{amscd}
\usepackage{amsmath,amsthm,amssymb,amscd,mathrsfs,picture}
\usepackage[arrow,matrix]{xy}
\input xy
\xyoption{all}

\usepackage{enumerate}

\theoremstyle{definition}
\newtheorem{defi}{Definition}[section]

\newtheorem{rem}[defi]{Remark}
\theoremstyle{plain}
\newtheorem{prop}[defi]{Proposition}

\newtheorem{lem}[defi]{Lemma}

\newtheorem{thm}[defi]{Theorem}

\newtheorem{cor}[defi]{Corollary}

\newcommand{\M}{\mathcal{M}}
\newcommand{\CHP}{\mathcal{CH}_{\geq 0}}

\newcommand{\OPER}{\mathcal{OPER}}
\newcommand{\SEQ}{\mathcal{SEQ}}
\newcommand{\oper}{\mathcal{O}}
\newcommand{\A}{\mathcal{A}}
\newcommand{\B}{\mathcal{B}}
\newcommand{\K}{\mathcal{K}}
\newcommand{\aoper}{\mathcal{P}}
\newcommand{\taoper}{\tilde{\mathcal{P}}}
\newcommand{\F}{\mathcal{F}}
\newcommand{\Hom}{H_*}
\newcommand{\Ho}{\mathbf{Ho}}

\newcommand{\W}{\mathcal{W}}

\newcommand{\U}{\mathcal{U}}

\newcommand{\tot}{\operatorname{Tot}}

\newcommand{\seq}{\mathcal{S}}

\newcommand{\field}{\mathbf{k}}
\newcommand{\Embbar}{\overline{Emb}}
\newcommand{\R}{\mathbb{R}}
\newcommand{\Poiss}{\mathrm{Poiss}}
\newcommand{\Vin}{V_{in}}

\newcommand{\In}{\mathrm{In}}
\newcommand{\T}{\mathcal{T}}

\begin{document}

\maketitle
\begin{center}
\vspace{-2mm}

\textit{MSC 2010 : 55U35\\
\vspace{1mm}
}
\end{center}
\begin{abstract}
Lambrechts, Turchin and Voli\'c \cite{LTV} proved the Bousfield-Kan type rational homology spectral sequence associated to the $d$-th Kontsevich operad collapses at $E^2$-page if $d\geq 4$.  The key of their proof is formality of the operad. 
In this paper, we simplify their proof using a model category of operads.  As byproducts  we obtain two new consequences. One is collapse of the spectral sequence in the case of $d=3$ (and the coefficients being rational numbers). The other says there is no non-trivial extension for the Gerstenhaber algebra structure on the spectral sequence.
\end{abstract}

\section{Introduction}
The $d$-th Kontsevich operad $\K_d$ is defined  as a certain compactification of the configuration space of ordered points in $\R^d$ for each $d\geq 1$ (see \cite{sinha}). It is weak equivalent to the little $d$-cubes operad, but it has the technical advantage that it admits a morphism of non-$\Sigma$ operads from the associative operad. So we may consider the associated cosimplicial space $\K_d^\bullet$ via the construction  of Gerstenhaber-Voronov \cite{GV} and McClure-Smith \cite{MS}. 
  Sinha \cite{sinha} proved the homotopy totalization of $\K_d^\bullet$ is weak homotopy equivalent to the space of 'long knots modulo immersions' $\Embbar_d$ if $d\geq 4$ (see \cite{sinha} or \cite{LTV} for the definition).  He also proved the Bousfield-Kan type homology spectral sequence associated to $\K_d^\bullet$ converges to the homology of  $\Embbar_d$ if $d\geq 4$. We simply call this spectral sequence Sinha's spectral sequence.\par
Lambrechts, Turchin, and Voli\'c \cite{LTV} proved Sinha's spectral sequence with rational coefficients collapses at $E^2$-page if $d\geq 4$. As the $E^2$-page is isomorphic to the Hochschild cohomology of the Poisson operad of degree $d-1$, we get a good algebraic presentation of the homology of $\Embbar_d$ by this collapse. The key of their proof is formality of the Kontsevich operad.\par
 The main purpose of this paper is to simplify their proof  using Quillen's theory of model categories. As byproducts we obtain some new consequences (the case of $d=3$ in Theorem \ref{Tcollapse} and Corollaries \ref{Cextension}, \ref{Cembbar}). To explain the situation more precisely, we prepare some notations and  terminologies. In the rest of the paper, an operad means a non-$\Sigma$ operad. Let $\CHP$ denote the category of non-negatively graded chain complexes over a fixed field $\field$ (with differentials decreasing degree) and $\OPER$ be the category of operads over $\CHP$. Let $\A\in\OPER$ denote the associative operad.

\begin{defi}\label{Dformality}
A morphism $f:\oper\to\aoper\in\OPER$ is called  \textit{a weak equivalence} if the chain map $f_n:\oper(n)\to\aoper(n)$ at each arity $n\geq 0$ is a quasi-isomorphism.  For an operad $\oper\in\OPER$ we define an operad $H_*(\oper)$ as follows: We put $H_*(\oper)(n)=H_*(\oper(n))$, where the right hand side is the usual homology group considered as a complex with the zero differential. The composition of $H_*(\oper)$ is induced by that of $\oper$. The construction $\oper\mapsto H_*(\oper)$ is natural for a morphism of  operads. We say a morphism  $f:\oper\to\aoper\in\OPER$ is \textit{relatively formal} if there exists a chain of commutative squares in $\OPER$:
\[
\xymatrix
{\oper\ar[d]_f&\oper_1\ar[d]\ar[l]\ar[r]&\cdots&\oper_N\ar[d]\ar[l] \ar[r]&  H_*(\oper)\ar[d] ^{H_*(f)}\\
\aoper &\aoper_1\ar[l]\ar[r] &\cdots & \aoper_N\ar[l]\ar[r] & H_*(\aoper)
}
\]
where each horizontal arrow is a weak equivalence. A \textit{multiplicative operad} is an operad $\oper$ equipped  with a morphism $\A\to \oper$. A morphism of multiplicative operads is a morphism of operads under $\A$. We say a multiplicative operad $f:\A\to\oper\in\OPER$ is \textit{multiplicatively formal} if it is relatively formal and one can take a chain of commutative squares connecting $f$ and $H_*(f)$  such that each horizontal morphism between sources is the identity (under the canonical identification $H_*(\A)=\A$).
\end{defi}

Let $C_*(\K_d)$ denote the chain operad of the Kontsevich operad with $\field$-coefficients.  A fixed linear embedding $\R\to \R^d$ induces a morphism $\K_1\to \K_d$ of operads. Composing this morphism with the morphism from the (topological) associative operad $\A^{top}$ to $\K_1$ which takes the unique point  to the configuration whose numbering is consistent with the order of $\R$, we obtain   a morphism $\A^{top}\to \K_d$. This morphism naturally induces a morphism $\A\to C_*(\K_d)$ in $\OPER$ and we regard $C_*(\K_d)$ as a multiplicative operad with this morphism. The following theorem is a special case of the main theorem of Lambrechts and Voli\'c \cite{LV}. 
\begin{thm}[\cite{LV} ]\label{Trelativeformality}
When $d\geq 3$ and $\field=\R$, the morphism $\A\to C_*(\K_d)$ is relatively formal. \qed
\end{thm}
(More precisely speaking,  in \cite{LV} the relative formality of Fulton-MacPherson operad is proved but  Theorem \ref{Trelativeformality}  immediately follows from it using the diagram (2.5) of \cite{LTV})\\
\indent The usual (absolute) formality of $C_*(\K_d)$ (or little $d$-cubes operad) was proved first by  Tamarkin \cite{tamarkin} for $d=2$ and later by Kontsevich \cite{kontsevich} for general $d$ (see \cite{LV} and \cite{HLTV} for detailed descriptions of Kontsevich's proof). The relative version in the above theorem was proved by Lambrechts and Voli\'c \cite{LV} verifying the quasi-isomorphisms sketched by Kontsevich commute with the morphisms of operads. But their diagram does not show the multiplicative formality of Kontsevich operads as it contains the chain operad of Stasheff's associahedra. \\ 
\indent To obtain the collapse from Theorem \ref{Trelativeformality}, 
the authors of \cite{LTV} introduced a partial generalization of  the construction of  Gerstenhaber-Voronov and McClure-Smith applicable to any morphism of operads. Though it is a very general construction and should have other applications,  their proof  is somewhat complicated and does not work for $d=3$.
On the other hand, as pointed out in \cite{LTV}, if $C_*(\K_d)$ is multiplicatively formal, the collapse easily follows from it. We  prove this is true (the proof is given in section \ref{Sformality}) :

\begin{thm}\label{Tmultiplicativeformality}
When $d\geq 3$ and $\field=\R$, $C_*(\K_d)$ is multiplicatively formal.
\end{thm}
When $d$ is equal to or greater than $4$, the following theorem is the main result of \cite{LTV}.
\begin{thm}\label{Tcollapse}
For $d\geq 3$, Sinha's spectral sequence with rational coefficients collapses at $E^2$-page. 
\end{thm}
\begin{proof}
By Theorem \ref{Tmultiplicativeformality}, there is a diagram of the following form:
\[
\xymatrix
{\A\ar[d]&\A\ar[d]\ar@{=}[l]\ar@{=}[r]&\cdots&\A\ar[d]\ar@{=}[l] \ar@{=}[r]&  \A\ar[d] \\
C_*(\K_d) &\aoper_1\ar[l]\ar[r] &\cdots & \aoper_N\ar[l]\ar[r] & H_*(\K_d)
}
\]
where each horizontal arrow is a weak equivalence.
As the construction of Gerstenhaber-Voronov and McClure-Smith is natural for morphisms of multiplicative operads, this diagram induces a chain of termwise quasi-isomorphisms of cosimplicial chain complexes as follows:
\[
\xymatrix{
C_*(\K_d^\bullet)
&
\aoper_1^\bullet \ar[l]\ar[r]
&
\cdots
&
\aoper_N^\bullet\ar[l]\ar[r] 
& 
H_*(\K_d^\bullet).}
\]
Here, $\aoper^\bullet$ denotes the cosimplicial chain complex associated to an operad $\aoper$, and a termwise quasi-isomorphism  is a morphism which induces a quasi-isomorphism at each cosimplicial degree.
In turn, this chain induces a chain of homomorphisms of spectral sequences
\[
\xymatrix{
E^r_{*,*}(C_*(\K_d^\bullet))
&
E^r_{*,*}(\aoper_1^\bullet ) \ar[l]\ar[r]
&
\cdots
&
E^r_{*,*}(\aoper_N^\bullet )\ar[l]\ar[r] 
& 
E^r_{*,*}(H_*(\K_d^\bullet) )}
\]
for each $r\geq 0$.
As all differentials of $H_*(\K^\bullet_d)$ are zero at each cosimplicial degree, the spectral sequence $\{E^r_{*,*}(H_*(\K^\bullet_d))\}_{r\geq 0}$ collapses at $E^2$-page. As a termwise quasi-isomorphism of cosimplicial complexes induces an isomorphism between $E^2$-pages, all arrows in the above chain are isomorphisms for each $r\geq 2$.  Thus we see the spectral sequence $\{E^r_{*,*}(C_*(\K^\bullet_d))\}_{r\geq 0}$ collapses at $E^2$-page.
\end{proof}
\begin{rem}
When $d=3$, it is \textit{not} known whether the Sinha's spectral sequence converges to the homology of $\Embbar_3$ but it is still worth  studying as its $E^2$-page is isomorphic to the $E^1$-page of Vassiliev's spectral sequence for long knots modulo immersions in $\R^3$ (see Turchin \cite{T}). In particular, its diagonal part (the part of total degree zero) is isomorphic to the space of all finite type invariants of framed long knots. See also Voli\'c \cite{V} for identification between the diagonal and invariants using Sinha's cosimplicial model and Bott-Taubes integral.\\
\indent   In \cite{LTV}, the authors deduce the collapse of Vassiliev's spectral sequence which converges to the homology of the space of long knots in $\R^d$ from collapse of Sinha's spectral sequence for each $d\geq 4$.  A similar argument does not seem to work for $d=3$ because in this case these two spectral sequences are not known to converge to the same module. 
\end{rem}
Besides simplification of the proof, Theorem \ref{Tmultiplicativeformality} has an immediate application to the multiplicative structure on the spectral sequence. 
The Hochschild cohomology $HH^*(\oper)$ of a chain multiplicative operad $\oper$ is by definition the homology of the total complex of the associated cosimplicial chain complex $\oper^\bullet$.  $HH^*(\oper)$ carries a natural Gerstenhaber algebra structure whose product and Lie bracket are defined similarly to those on the Hochschild cohomology of an associative algebra (see \cite{GV, salvatore}). 
\begin{cor}\label{Cextension}
When $\field=\R$ and $d\geq 3$, there exists an isomorphism of Gerstenhaber algebras: 
\[
 HH^*(C_*(\K_d))\cong HH^*(\Poiss_{d-1}).
\]
Here, $\Poiss_{d-1}$ is the Poisson operad of degree $d-1$ (see \cite[Definition 4.10]{sinha}).
\end{cor}
\begin{proof}
 As the construction of the Gerstenhaber algebra from a multiplicative operad is natural for morphisms of multiplicative operads,  the chain of termwise quasi-isomorphisms in the proof of Theorem \ref{Tcollapse} induces an isomorphism of Gerstenhaber algebras $HH^*(C_*(\K_d))\cong HH^*(H_*(\K_d))$. Since the operad $H_*(\K_d)$ is isomorphic to $\Poiss_{d-1}$ as a multiplicative operad, we have proved the corollary. 
\end{proof}
This corollary says there is no extension problem in Sinha's spectral sequence as the right hand side is isomorphic to the $E^2$-page with the induced operations. The utility of formality for the extension problem was pointed out by Salvatore \cite{salvatore} (but a proof was not given).\par
McClure and Smith \cite{MS} invented a topological version of  the above construction. For a topological multiplicative operad, they defined a little squares action on its (homotopy) totalization. In particular, for $d\geq 4$ the homology 
$H_*(\Embbar_d)\cong H_*(\widetilde{\tot}(\K_d^\bullet))$ carries an induced Gerstenhaber structure whose product and bracket are given by the Pontryagin product and  Browder operation. (Here $\widetilde{\tot}$ denotes the homotopy totalization i.e., the homotopy limit over the category of simplices $\Delta$.) We obtain an algebraic interpretation of this 'topological' Gerstenhaber algebra:
\begin{cor}\label{Cembbar}
When $\field=\R$ and $d\geq 4$, there exists an isomorphism of Gerstenhaber algebras:
\[
H_*(\Embbar_d)\cong HH^*(\Poiss_{d-1}).
\]
\end{cor}
\begin{proof}
Combine Corollary \ref{Cextension} with Sakai \cite[Theorem 4.6]{sakai} or \cite[Proposition 22]{salvatore}.
\end{proof}
\begin{rem}
Songhafouo-Tsopm\'en\'e \cite{paul} also obtained the results stated above  independently and simultaneously.
\end{rem}
\section{Proof of Theorem \ref{Tmultiplicativeformality}}\label{Sformality}
Besides Theorem \ref{Trelativeformality}, the other key to the proof of Theorem \ref{Tmultiplicativeformality} is the following:
\begin{thm}[\cite{hinich,  spitzweck, BM, muro}]\label{Tmodel}
The category $\OPER$ of non-$\Sigma$-operads over $\CHP$ admits a left proper model category structure where
\begin{itemize}
\item  weak equivalences are those defined in Definition \ref{Dformality}, and 
\item fibrations are those morphisms $f:\oper\to \aoper$ such that  for each $n\geq 0$ and $k\geq 1$, the linear map  $f_{n,k}:\oper(n)_k\to \aoper(n)_k$ at arity $n$ and degree $k$ is an epimorphism.
\end{itemize}
\end{thm}

Our notion of a model category is that of Hovey \cite{hovey}.  Recall that a model category $\M$ is said to be \textit{left proper} if a pushout of a weak equivalence by a cofibration is also a weak equivalence. Theorem \ref{Tmodel} is not  new one. For the case of $\Sigma$-operads, existence of a (semi) model category structure was proved first by Hinich \cite{hinich} for chain complexes and later by Spitzweck \cite{spitzweck} and Berger-Moerdijk \cite{BM} for general model categories, and left properness was also proved in \cite{spitzweck, BM}. For the case of non-$\Sigma$ operads, Muro \cite{muro} proved existence of  a model category structure for general model categories. In our simple case the proof is somewhat shorter.  For the reader's convenience we give a proof of Theorem \ref{Tmodel} in section \ref{Smodel}. \par
For a model category $\M$ and a morphism $f:X\to Y\in \M$, a Quillen adjoint pair
\[
\xymatrix{P_f:X/\M\ar[r]<1mm>&Y/\M: U_f\ar[l]<1mm>}
\] 
 between under categories (with the comma model structures, see the paragraph below Proposition 1.1.8 of \cite{hovey}) is defined by $P_f(Z)=Y\cup_XZ$ and $U_f(Z)=$ the composition $X\to Y\to Z$. The following proposition is well-known and can be easily proved using \cite[Corollary 1.3.16]{hovey}.
\begin{prop}\label{Pleftproper}
Under the above notations, if $\M$ is  left proper, for any weak equivalence $f$ the induced adjunction $(P_f,U_f)$ is a Quillen equivalence. In particular, the derived adjunction $(\mathbb{L}P_f,U_f)$ induces an equivalence between the homotopy category.\qed
\end{prop}
An operad weak equivalent to $\A$ is called an $A_\infty$-operad. Let $\Ho(\M)$ denote the homotopy category of  a model category $\M$.
\begin{lem}\label{Lhomotopyclass}
(1) Let $\B_0$ and $\B_1$ be two $A_\infty$-operads and  suppose $\B_0$ is cofibrant. There exists a bijection $[\B_0,\B_1]\cong \field^{\times}$. Here, $[-,-]$ denotes the set of (left or right) homotopy classes of  morphisms. If one fixes a morphism $f:\B_0\to \B_1$, the bijection is given by $\field^{\times}\ni a\mapsto a*f\in[\B_0,\B_1]$, $(a*f)_n=a^{n-1}f_n$.\\
\indent (2) Let 
\[
\xymatrix{
\B_0\ar[d]^{f_1}&\B_0\ar[d]^{f_2}\\
\B_1\ar[r]^{\alpha}\ar[d]^{g_1}&\B_2\ar[d]^{g_2}\\
\oper_1\ar[r]^\beta&\oper_2}
\] 
be a commutative diagram in $\OPER$ where $\B_0$, $\B_1$ and $\B_2$ are $A_\infty$-operads, and $\B_0$ is cofibrant, and $\beta$ is a weak equivalence. Then the compositions $g_1\circ f_1$ and $g_2\circ f_2$ are isomorphic as objects of $\Ho(\B_0/\OPER)$.
\end{lem}
\begin{proof}
(1) As any object of $\OPER$ is fibrant, by  homotopy invariance of the set of homotopy classes, we may replace $\B_1$ with the associative operad $\A$. As a morphism $f:\B_0\to \A$ uniquely factors through the morphism $\Hom(f):\Hom(\B_0)\to\Hom(\A)\cong \A$, the set $[\B_0,\A]$ is bijective to the set of endomorphisms on $\A$. This latter set is bijective to $\field^{\times}$ since an endomorphism of $\A$ is determined by its image of a generator of $\A(2)$.\par
(2) We shall consider the case where $\alpha$ and $\beta$ are the identities. If $f_1$ and $f_2$ are homotopic, by standard properties of left and right homotopies, $g_1\circ f_1$ and $g_2\circ f_2$ ($=g_1\circ f_2$) are right homotopic.  This  implies the claim by definition. So we may assume $f_2=a*f_1$ for some $a\in\field^{\times}$ by part 1. Define a morphism $\phi_a:\oper_1\to\oper_1$ as $\phi_{a,n}=a^{n-1} :\oper_1(n)\to \oper_1(n)$. Clearly $\phi_a$ is an isomorphism between $g_1\circ f_1$ and $g_2\circ f_2$ in $\B_0/\OPER$, hence in $\Ho(\B_0/\OPER)$.\\
\indent We shall consider the general case. Clearly $g_1\circ f_1$ and $g_2\circ \alpha\circ f_1$ are isomorphic in $\Ho(\B_0/\OPER)$. By applying the above case to $f_2$ and $\alpha \circ f_1$, we get the claim in the general case. 
\end{proof}

\textit{Proof of Theorem \ref{Tmultiplicativeformality}}.\quad By relative formality (Theorem \ref{Trelativeformality}), there exists a commutative diagram in $\OPER$:
\[
\xymatrix
{\A\ar[d]^{g}&\B_1\ar[d]^{g_1}\ar[l]\ar[r]&\cdots&\B_N\ar[d]^{g_N}\ar[l] \ar[r]&  \A\ar[d]^{H_*(g)} \\
C_*(\K_d) &\oper_1\ar[l]\ar[r] &\cdots & \oper_N\ar[l]\ar[r] & H_*(\K_d),
}
\]
where all horizontal arrows are weak equivalences.  Let $f:\B_0\to \A$ be a cofibrant replacement of the associative operad. We can pick a morphism $f_i:\B_0\to \B_i\in\OPER$ for each $i=1,\dots,N$ as each $B_i$ is a  (fibrant) $A_\infty$-operad. So we obtain the following diagram:
\[
\xymatrix
{\B_0\ar[d]^{f}&\B_0\ar[d]^{f_1}&\cdots&\B_0\ar[d]^{f_N} &  \B_0\ar[d]^{f} \\
\A\ar[d]^g&\B_1\ar[d]^{g_1}\ar[l]\ar[r]&\cdots&\B_N\ar[d]^{g_N}\ar[l] \ar[r]&  \A\ar[d]^{H_*(g)} \\
C_*(\K_d) &\oper_1\ar[l]\ar[r] &\cdots & \oper_N\ar[l]\ar[r] & H_*(\K_d).
}
\]

By applying the part 2 of Lemma \ref{Lhomotopyclass} to each part of this diagram, we see $g\circ f$ and $g_1\circ f_1$ and $g_2\circ f_2$ and ... $H_*(g)\circ f$ are all  isomorphic in $\Ho(\B_0/\OPER)$. In other words, $U_f(C_*(\K_d))$ and $U_f(H_*(\K_d))$ are isomorphic in $\Ho(\B_0/\OPER)$. By this and Proposition \ref{Pleftproper} we have isomorphisms  
$C_*(\K_d)\cong \mathbb{L}P_fU_f(C_*(\K_d))\cong\mathbb{L}P_fU_f(H_*(\K_d)))\cong H_*(\K_d)$ in $\Ho(\A/\OPER)$. This implies $C_*(\K_d)$ and $H_*(\K_d)$ can be connected by a chain of weak equivalences under $\A$, which means $C_*(\K_d)$ is multiplicatively formal.\qed
\section{Proof of Theorem \ref{Tmodel}}\label{Smodel}
To prove Theorem \ref{Tmodel}, we use \cite[Theorem 2.1.19]{hovey}. So we need two sets $I$ and $J$ of morphisms of $\OPER$, which play the roles of sets of  generating cofibrations and trivial cofibrations, respectively. To define $I$ and $J$, we use
 the free construction in $\OPER$, and to describe this construction, we shall recall languages of tree.\par 
A \textit{tree} is a finite connected acyclic graph. Let $T$ be a tree and
 $\phi:|T|\to \R\times[0,1]$ be an embedding of
 the geometric realization of $T$ such that $Im(\phi)\cap \R\times\{0\}$ consists of only one vertex (or 0-cell), 
which we call the \textit{root} of $T$ and $Im(\phi)\cap \R\times\{1\}$ consists of univalent vertices. 
Let $n\geq 0$ be an integer and $\alpha:\{1,\dots, n\}\to Im(\phi)\cap\R\times\{1\}$ be an order-preserving monomorphism, where the linear order on $Im(\phi)\cap\R\times\{1\}$ is induced by the usual order on $\R\times\{1\}=\R$. We call a vertex in $Im(\alpha)$ a \textit{leaf} of $T$ and a vertex in $Im(\phi)\cap\R\times\{1\}-Im(\alpha)$ a \textit{null vertex} of $T$.  For an edge $e$ of $T$, the vertex of $e$ further from the root is called the \textit{source} of $e$, and the other is called the \textit{target}. 
\begin{defi}

For each $n\geq 0$ consider isotopy classes of triples $(T,\phi,\alpha)$ which satisfy the above conditions, where an isotopy is assumed to respect the map $\alpha$.  We call such an isotopy class $\{(T,\phi,\alpha)\}$ a \textit{regular planer $n$-tree} if  each vertex in $Im(\phi)\cap \R\times (0,1)$ is at least bivalent.
By abuse of notations, a regular planer $n$-tree is denoted by the same notation as the underlying tree. \par
Let $T$ be a regular planer $n$-tree and $v$ be a vertex of $T$. We define a number $\In(v)$ as $0$ if $v$ is a null vertex, and as the number of the  edges whose targets are $v$ otherwise.  The set of vertices which are not leaves is denoted by $\Vin(T)$. 
The \textit{level} of a vertex $v$   is one less than the number of vertices on the shortest path connecting the root and $v$. For example, the level of the root is 0.   We put
\[
\Vin^0(T)=\{v\in\Vin(T)\mid  \text{ the level of $v$ is even}\},\quad \Vin^1(T)=\Vin(T)-\Vin^0(T).
\]
We say $T$ is \textit{odd} if the level of each vertex in $Im(\alpha)$ is odd. 
Let $\T_n$ (resp. $\T^1_n$) denote the set of all regular planer $n$-trees (resp. odd regular planer $n$-trees).
\end{defi}
\begin{rem}
For each $n\geq 0$, $\T_n$ is bijective to the set of all isomorphism classes of planted planer trees with $n$ leaves defined in \cite[Definition 3.4]{muro}.
\end{rem}
Let $\SEQ$ be the category of sequences in $\CHP$.  An object of $\SEQ$ is a sequence $\seq(0),\seq(1),\dots$ of chain complexes, and  a morphism is a sequence of chain maps. 
The free construction (or free functor) $\F:\SEQ\longrightarrow \OPER$, i.e. the left adjoint of the forgetful functor $\U:\OPER\longrightarrow \SEQ$ is defined by 
\[
\F(\seq)(n)=
\field \cdot \delta_{1,n}\oplus\bigoplus\limits_{T\in\T_n}\bigotimes\limits_{v\in\Vin(T)}\seq(\In(v)).
\]
Here, $\field\cdot \delta_{1,n}$ is the module generated by a formal unit if $n=1$, and the zero module otherwise.\par 
 
To define  the sets of morphisms $I$ and $J$, we shall recall a set of generating (trivial) cofibrations of $\CHP$. For  $p\geq 1$ we define a complex $D^p$ as follows: $D^p_l=\field$ if $l=p,p-1$, $D^p_l=0$ otherwise. The differential $d_p$ is the identity. For $p\geq 0$ we define another complex $S^p$ by $S^p_p=\field$ and $S^p_l=0$ for $l\not=p$. Let $i^p:S^{p-1}\to D^p$ be the chain map which is the identity on $S^{p-1}_{p-1}$, and $j^p:0\to D^p$ be the unique chain map.  Put $I_0=\{i^p\mid p\geq 1\}$, and $J_0=\{j^p\mid p\geq 1\}$. The following proposition is well known and can be proved by a way analogous to the proof of \cite[Theorem 2.3.11]{hovey}.
\begin{prop}\label{Pchain}
$\CHP$ admits a cofibrantly generated model category structure with $I_0$ (resp. $J_0$) being a set of generating cofibrations (resp. trivial cofibrations), where 
\begin{itemize}
\item weak equivalences are  quasi-isomorphisms, 
\item fibrations are those morphisms $f:C\to D$ such that  the linear map $f_k:C_k\to D_k$ is an epimorphism for each $k\geq 1$.\qed
\end{itemize}
\end{prop}
Let $D^{p,q}$ and $S^{p,q}$ be two objects of $\SEQ$ defined by $D^{p,q}(q)=D^p$, $D^{p,q}(n)=0$ for $n\not=q$, and $S^{p,q}(q)=S^p$, $S^{p,q}(n)=0$ for $n\not=q$, respectively.   
Let $i^{p,q}:S^{p-1,q}\to D^{p,q}$ and $j^{p,q}:0\to D^{p,q}$ be the morphisms induced by $i^p$ and $j^p$ respectively. Put $I_1=\{i^{p,q}\mid p\geq 1, q\geq 0\}$ and $J_1=\{j^{p,q}\mid p\geq 1,q\geq 0\}$.\par
 We define two sets of morphisms $I$ and $J$ as the image of $I_1$ and $J_1$ by $\F$, respectively.\par
\vspace{\baselineskip}

\textit{Proof of Theorem \ref{Tmodel}.}  We apply   \cite[Theorem 2.1.19]{hovey} to the sets $I$ and $J$ defined above and the class of weak equivalences given in Theorem \ref{Tmodel}. We must verify the six conditions stated in the theorem.  The first condition (2-out-of-3 and closedness under retraction of $\W$) is clear. The second and third conditions (smallness of the domains of $I$ and $J$) follow from \cite[Lemma 2.3.2]{hovey} and adjointness of the pair $(\F,\U)$. By adjointness, the class $I$-$inj$ (resp. $J$-$inj$) is equal to the class of morphisms $f:\oper\to\oper'$ such that $\U(f):\U(\oper)\to\U(\oper')$ is $I_1$-$inj$ (resp. $J_1$-$inj$). This and Proposition \ref{Pchain} imply the  fifth and sixth  conditions ($I$-$inj\subset\W\cap J$-$inj$ and $\W\cap J$-$inj\subset I$-$inj$).\\
\indent We shall prove the forth condition ($J$-$cell\subset \W\cap I$-$cof$). $J$-$cell\subset I$-$cof$ is clear by adjointness. To prove $J$-$cell\subset \W$, as quasi-isomorphisms are closed under transfinite composition, it is enough to prove a pushout by a morphism in $J$ is in $\W$. Take a morphism  $\F(j^{p,q}):\F(0)\to\F(D^{p,q})\in J$ and an operad $\oper$. The pushout $\aoper=\F(D^{p,q})\cup_{\F(0)}\oper $ ($=\F(D^{p,q})\sqcup \oper$) has the following presentation.
\[
\aoper(n)=\bigoplus\limits_{T\in\T_n^1} \left(\bigotimes\limits_{v\in \Vin^0(T)}
\oper(\In(v))\otimes\bigotimes\limits_{v\in \Vin^1(T)}D^{p,q}(\In(v))\right).
\]
(In this presentation, the unit of $\oper$ serves as the unit of $\aoper$, and for $x,y\in D^{p,q}$, a composition $x\circ_iy$ is equal to $((1\circ x\circ_i1)\circ_i y)\circ 1^{\otimes m}$ for some $m$, so we do not need 'even' trees or other partitions of the set $\Vin(T)$.) As the tensor product over a field preserves quasi-isomorphisms, we see the pushout morphism $\oper\to \aoper$ is in the class $\W$ from this presentation. \\
\indent We shall prove left properness. As any cofibration is a retract of a relative $I$-cell, it is enough to prove a pushout by a generating cofibration preserves weak equivalences. Take a morphism $\F(i^{p,q}):\F(S^{p-1,q})\to\F(D^{p,q})$. Let $f:\oper\to \oper'$ be a weak equivalence and $g:\F(S^{p-1,q})\to \oper $ be a morphism. As a graded operad, i.e., if we forget the differentials, the pushout ${\taoper}_{\oper}=\oper\cup_{\F(S^{p-1,q})}\F(D^{p,q})$ has a presentation analogous to the above presentation of $\aoper$. It is given by replacing $D^{p,q}$ with $S^{p,q}$ in the above one.
By the Leipnitz rule, the differential is determined by its restrictions to $\oper$  and to $S^{p,q}$. On $\oper$ it is equal to the original differential of $\oper$, and on $S^{p,q}$ it is given by the composition $S^{p,q}(q)_p=D^p_p\xrightarrow{d^p}D^{p}_{p-1}=S^{p-1,q}(q)_{p-1}\xrightarrow{g}\oper(q)_{p-1}$. What we have to prove is that the induced morphism $\taoper_{\oper}\to \taoper_{\oper'}$ is a weak equivalence. For each $l\geq 0$ let $F_\oper^l\subset \taoper_{\oper}$ be the subsequence which is spanned by the summands corresponding to regular planer trees with $\sharp\Vin^1(T)\leq l$. As $F_\oper^{l+1}(n)/F_\oper^l(n)$ is isomorphic to a sum of  tensors of $\oper(m)$'s and $S^p$'s (as chain complexes), the induced morphism $F_\oper^{l+1}(n)/F_\oper^l(n)\to F_{\oper'}^{l+1}(n)/F_{\oper'}^l(n)$ is a quasi-isomorphism. By an inductive argument using long exact sequence, we see the morphism $F_\oper^{l}(n)\to F_{\oper'}^{l}(n)$ is a quasi-isomorphism for each $l\geq 0$. As $\taoper_\oper=\cup_{l}F_\oper^l$, we obtain the desired claim.
\qed \\

\textbf{Acknowledgement.} The author is grateful to Masana Harada for many valuable comments to improve presentations of the paper and other mathematical and non-mathematical supports.  He is also grateful to Keiichi Sakai for kindly answering his questions, telling him about the paper \cite{salvatore} and giving some valuable comments for the first version of this paper.
{\sc Department of Mathematics, Kyoto University, 606-8502, Japan}\\
\textit{E-mail adress} : {\tt moriyasy@math.kyoto-u.ac.jp}

\end{document}